\documentclass{article}

\usepackage[english]{babel}

\usepackage[letterpaper,top=2cm,bottom=2cm,left=3cm,right=3cm,marginparwidth=1.75cm]{geometry}
\usepackage{textcomp} 
\usepackage{amsmath, amsthm, amsfonts, amssymb}
\usepackage{marvosym}
\usepackage{mathrsfs}
\usepackage{graphicx}
\usepackage[colorlinks=true, allcolors=blue]{hyperref}
\usepackage{mathtools}
\newcommand{\E}{{\mathbb{E}}}
\newcommand{\R}{{\mathbb{R}}}
\newcommand{\D}{{\mathbb{D}}}

\newcommand{\N}{{\mathbb N}}

\newcommand{\ve}{\varepsilon}

\newtheorem{thm}{Theorem}[section]
\newtheorem{lem}{Lemma}[section]
\newtheorem{example}{Example}[section]
\newtheorem{remark}{Remark}[section]

\newtheorem{defn}{Definition}[section]

\begin{document}
\begin{center}
\textbf{  
\begin{Large}
    Universal generalized functionals and finitely absolutely continuous measures on Banach spaces 
\end{Large}
 }
\end{center}

\begin{center}
   \textbf{A. A. Dorogovtsev} \\
   Institute of Mathematics, National Academy of Sciences of Ukraine, Ukraine \\
E-mail address : \textit{andrey.dorogovtsev@gmail.com }

\end{center}

\begin{center}
\textbf{Naoufel Salhi} \\
Laboratory of Stochastic Analysis and Applications, Department of Mathematics, Faculty of Sciences of Tunis, University of Tunis El Manar, Tunisia\\
E-mail address : \textit{salhi.naoufel@gmail.com}
\end{center}
\begin{center}
\textit{To the memory of professor Habib Ouediane}
\end{center}
\begin{abstract}
In this paper we collect several examples of convergence of functions of random processes to generalized functionals of those processes. We remark that the limit is always finitely absolutely continuous with respect to Wiener measure. We try to unify those examples in terms of convergence of probability measures in Banach spaces. The key notion is the condition of uniform finite absolute continuity.
\end{abstract}

\begin{flushleft}

\end{flushleft}

\newpage

\section*{Introduction}
In this article we discuss the generalized functionals from the stochastic processes. The idea of generalized functionals is closely related to the investigation of geometric properties of the random processes. Simple but important examples are the Rice formula \cite{Adler} and the local time of Wiener process \cite{Markus}. 
\begin{example}
    Suppose that $\xi (t), \, t\in [0,1]$ is a centered Gaussian process with the smooth covariance. Then, for fixed level $c$, the expected number of upcrossings of the level $c$ by the process $\xi$ is equal \cite{Adler}  to the following expression 
\begin{equation}
    \label{Rice}
    \int_0^1 \int_0 ^{+\infty} xq_t(c,x)dxdt  
\end{equation}                                              
Here $q_t(x)$ is the joint distribution of $\xi(t)$ and $\xi'(t)$. Formally, expression \eqref{Rice}  can be obtained easily. Every upcrossing is associated with the formal expression                  \begin{equation}
    \label{RiceDirac}
    \delta_c \big( \xi(t) \big) \mathbf{1}_{(0, + \infty )} \big( \xi'(t) \big) \xi'(t) 
\end{equation}                                           
Keeping in mind that the result of action of $\delta_c$ on the continuous bounded function is the value of the function at the point $c$ we easily get the Rice formula                                            
\begin{equation}
    \label{RiceFormula}
   \int_0^1 \E \Big( \delta_c \big( \xi(t) \big) \mathbf{1}_{(0, + \infty )} \big( \xi'(t) \big) \xi'(t)   \Big) dt= \int_0^1 \int_0 ^{+\infty} xq_t(c,x)dxdt  
\end{equation}                                            
Such approach together with the formal expression was proposed by M. Kac \cite{Kac}. But how to bring the rigorous sense to the expression \eqref{RiceDirac}?   

\end{example}
 The same problem arises under the consideration of the local time of Brownian motion.
 \begin{example}
     Let $w(t),\, t\in [0,1]$ be a standard Wiener process. It is well known \cite{Markus} that for any $x\in \R$ there exists the local time which $w$ spends at the infinitesimal small neighborhood of $x$, i.e. 
     $$ \ell (x)=L_2-\lim _{\ve \to 0^+} \dfrac{1}{2\ve} \int_0^1 \mathbf{1}_{[x-\ve, x+\ve]} \big( w(t) \big) dt. $$
     Keeping in mind that in the sense of generalized functions 
     $$ \lim _{\ve \to 0^+} \dfrac{1}{2\ve} \mathbf{1}_{[x-\ve, x+\ve]} =\delta_x $$
     one can write formally 
     $$ \ell(x)=\int_0^1 \delta_x \big( w(t) \big) dt .$$
     Such expression is very useful for production of new formulas for local time which can then be proved rigorously. For example, occupation formula for bounded measurable function $f$ 
     \begin{align*}
         \int_{\R } f(x) \ell(x)dx 
         &=\int_{\R } f(x)\int_0^1 \delta_x \big( w(t) \big) dtdx \\
         &=\int_0^1 \Big(  \int_{\R } f(x)  \delta_x \big( w(t) \big)dx \Big) dt \\
         &=\int_0^1 f\big( w(t) \big) dt,
     \end{align*}
     or Kac moment formula 
     \begin{align*}
         \E \ell(x)^n 
         &=\E \int_0^1 \cdots\int_0^1 \delta_x \big( w(t_1) \big) \cdots \delta_x \big( w(t_n) \big) dt_1\cdots dt_n  \\
         &=n! \int_{ 0\leqslant t_1 \cdots \leqslant t_n \leqslant 1  } \E \Big( \delta_x \big( w(t_1) \big) \cdots \delta_x \big( w(t_n) \big) \Big) dt_1\cdots dt_n  \\
         &=n! \int_{ \Delta_n  } \dfrac{1}{ (2\pi)^{ \frac{n}{2} }  } e^{-\frac{x^2}{2t_1}} \dfrac{1}{\sqrt{t_1}} \dfrac{1}{\sqrt{t_2-t_1}} \cdots \dfrac{1}{\sqrt{t_{n}-t_{n-1}}}  dt_1\cdots dt_n ,
     \end{align*}
     where $\Delta_n=\{(t_1,\cdots,t_n) \, : \, 0\leqslant t_1 \cdots \leqslant t_n \leqslant 1\}$. In this case also it is useful to create a formal definition of $\delta_x \big( w(t) \big) $ and the rules of manipulation with it. Such formal definition was done in \cite{FHSW, Kuo} and became to be a partial case of the notion of generalized Wiener function. 
 \end{example}
   In order to define generalized Wiener functions we need to introduce the family of Sobolev spaces $\D ^{2,\gamma}$ over Wiener space (see for example \cite{Mall, Sugita} for more details). First of all, let $w(t)=(w_1(t),\cdots, w_d(t))$ be a $d$-dimensional Brownian motion and denote by $\sigma(w)$ the $\sigma$-field generated by it. It is known that every square integrable Wiener random variable $\eta \in L^2\big( \Omega, \sigma(w), \mathbb{P} \big)$ has an It\^{o}-Wiener expansion \cite{AAD} which consists on the $L^2$-convergent series 
   $$ \eta =\sum_{k=0}^{\infty} I_k(f_k) $$
   where $I_k(f_k)$ denotes a $k$-multiple It\^{o} stochastic integral of the deterministic and symmetric square integrable kernel $f_k$. Now let $\gamma \in \R$. The Sobolev space $\D^{2,\gamma}$ is the completion of the following space :
$$ \left\lbrace   \eta=\sum_{k=0}^n I_k (f_k) \in L^2\big( \Omega, \sigma(w),\mathbb{P}\big),\, n\in \N\right\rbrace  $$
with respect to the norm : 
$$ \big \| \eta\big \| _{2,\gamma} ^2=\sum_{k=0}^{n} (k+1)^{\gamma } \, \E \, I_k(f_k)^2 \,.  $$
If $0<\gamma_1 <\gamma_2 $ then the following inclusions are true
$$  \D^{2,\gamma_2} \subset \D^{2,\gamma_1} \subset \D^{2,0}=L^2\big( \Omega, \sigma(w),\mathbb{P}\big) \subset \D^{2,-\gamma_1} \subset \D^{2,-\gamma_2} . $$
Moreover, for any real number $\gamma$ the space $\D^{2,-\gamma} $ is the dual space of $\D^{2,\gamma} $. When $\gamma <0$, the elements of $\D^{2,\gamma}$ are called generalized Wiener functionals. The spaces $\D^{2,\pm\infty}$ are defined respectively as projective and inductive limits  
$$ \D^{2,+\infty}=\cap_{\gamma >0} \D^{2,\gamma}\,,\; \D^{2,-\infty}=\cup _{\gamma >0} \D^{2,-\gamma}\,. $$
The probability space $\big( \Omega, \sigma(w),\mathbb{P}\big) $ can be replaced by the classical Wiener space 
$$\Bigg( \mathscr{C}_0\big( [0,1],\R^d\big),\mathcal{B}\Big( \mathscr{C}_0\big( [0,1],\R^d\big)  \Big),\mu_0 \Bigg)$$ 
where 
$$\big( \mathscr{C}_0\big( [0,1],\R^d\big)=\left\lbrace \omega :[0,1]\to \R^d,\, \text{continuous},\, \omega(0)=0\right\rbrace  $$ 
is endowed with the Borel $\sigma$-field $\mathcal{B} \Big(  \mathscr{C}_0\big( [0,1],\R^d\big)  \Big)$ generated by the supremum norm and with the standard Wiener measure $\mu_0$. \\
Let $\gamma >0$. A test function $F\in \D^{2,\gamma} $ is said to be positive if 
$$ F(\omega) \geqslant 0 \quad \mu_0 - \text{a.e.} $$
A generalized Wiener function $\eta \in \D^{2,-\gamma}$ is said to be positive if the bilinear pairing with any positive test function $F\in \D^{2,\gamma} $ is non-negative :
$$ (\eta, F) \geqslant 0.$$
A Sugita theorem \cite{Sugita} states that any positive generalized Wiener function can be represented by a measure on the Wiener space. 
\begin{thm}\label{sugita} \cite{Sugita}
If a generalized Wiener functional $\eta \in \D^{2,-\infty}$ is positive then there exists a unique finite positive measure $\theta $ on $\mathscr{C}_0\big( [0,1],\R^d\big)$ such that 
$$\forall \; F\in \mathcal{F}\mathscr{C}_b^{\infty}\Big( \mathscr{C}_0\big( [0,1],\R^d\big)  \Big) , \quad(\eta, F)=\int_{ \mathscr{C}_0\big( [0,1],\R^d\big) } F(\omega )\,\theta (d\omega) \, .$$
Here $\mathcal{F}\mathscr{C}_b^{\infty}\Big( \mathscr{C}_0\big( [0,1],\R^d\big)  \Big)$ denotes the set of Wiener test functions $F\in \D^{2,+\infty} $ of the form 
$$F(\omega)=f(\ell_1(\omega),\cdots,\ell_n(\omega) )$$ 
where $f\in \mathscr{C}_b^{\infty}(\R^d)$ and $\ell_1,\cdots,\ell_n \in \Big( \mathscr{C}_0\big( [0,1],\R^d\big) \Big) ^*$.\\
\end{thm} 
One of the possible applications of the mentioned theory is the study of the geometry of the trajectories of multi-dimensional Brownian motion. Such investigations are related to the mathematical model of free linear polymer \cite{Bolthausen, Hollander}. In this model, a trajectory of Wiener process is treated as an instant conformation of the long linear polymer molecule. Then, the excluded volume effect leads to the necessity to study self-intersections. In this way Evans measure arises \cite{Evans}. Since the theory is useful and well-developed for Wiener process then it is natural to ask about the same constructions and properties for processes different from Wiener process. A lot of attempts were done in this direction (see, for instance, \cite{Markus, Rosen, LeGall}) but here we mention the class of processes introduced by A.A. Dorogovtsev and called Gaussian integrators. 
\begin{defn}\cite{Dor} \label{def}
    A (one dimensional) centered Gaussian process $ x(t) , t \in [0,1] $ is said to be an integrator if there exists a constant $C>0$ such that for any arbitrary partition $ 0=t_{0} < t_{1}< \cdots < t_{n}=1 $ and real numbers $ a_{0} , \cdots ,a_{n-1} $ : 
\begin{equation}
\label{Ineq}
\mathbb{E}\left[ \sum_{k=0}^{n-1}a_{k}( x(t_{k+1})-x(t_{k})) \right]^{2}  \; \leq \; C\sum_{k=0}^{n-1}a_{k}^{2}(t_{k+1}-t_k) . 
\end{equation}
\end{defn}
For such processes some singular functionals were considered (see, for example, \cite{DOGS, DOS}). Recently, in \cite{DS24} the large deviation principle for the measure which is corresponding to generalized functional of self-intersection was obtained. In view of mentioned works we face the following general problem which is main object of discussions in this paper. Let us formulate it in the abstract form. \\
  
  Let $B$ be a real separable Banach space with the norm $\| \cdot \|$. All the measures on $B$ are supposed to be defined on the Borel $\sigma$-field $\mathcal{B}$. Also, all functions on $B$ are supposed to be Borel measurable. Consider a family $\Phi _\ve,\, \ve >0$ of functions on $B$ such that for every $\ve >0$, $\Phi_\ve : B \to \R$ has a continuous Frechet derivative and is bounded on $B$ together with its derivative. Suppose that $\mu$ is a probability measure on $B$ with all finite moments of the norm and such that finite-dimensional polynomials are dense in $L_2\big( B,\mathcal{B},\mu \big)$. Then, for every $\ve >0$, $\Phi_\ve$ has an expansion via orthogonal polynomials 
$$\Phi_\ve = \sum_{n=0}^{\infty} I_n^\ve.$$
Suppose that, for every $ \ve >0$, $\Phi_\ve \geqslant 0$ and for every $n\geqslant 0$ there exists a limit 
$$I_n^0=L_2-\lim _{\ve \to 0 } I_n^\ve.$$
Then the formal series 
$$\sum_{n=0}^\infty I_n^0$$
can be considered as a non-negative generalized functional on the measurable space $\big( B,\mathcal{B},\mu \big)$. In certain cases it can be associated with some new measure $\nu$ on $\mathcal{B}$ (as in Sugita theorem). In this article we discuss the following problem. When for the family $\{\Phi_\ve, \ve >0\}$ there exists a set of probability measures $\mathcal{M}$ such that for every $\mu\in \mathcal{M}$, $\Phi_\ve$ converges to a generalized functional on $\big( B, \mathcal{B}, \mu \big)$? Consider an example of such situation. 
\begin{example}
    Let $B=\mathscr{C}_0\big( [0,1],\R\big)$ be the $1$-dimensional Wiener space. Define the family $\Phi_\ve$ as follows : 
    $$\forall \; f\in \mathscr{C}_0\big( [0,1],\R\big), \quad \Phi_\ve (f)=\dfrac{1}{\sqrt{2\pi \ve}} e^{ -\frac{f(1)^2}{2\ve} }.  $$
    Now consider as a measure $\mu$ the standard Wiener measure $\mu_0$ on $\mathscr{C}_0\big( [0,1],\R\big)$. Then, (see, for example, \cite{DS23}) $\Phi_\ve$ converges when $\ve \to 0^+$ to a generalized functional which has an It\^{o}-Wiener expansion 
    $$ \Phi_0(f)=\sum_{n=0}^\infty \dfrac{1}{n!\sqrt{2\pi}} H_n(0) H_n\big( f(1) \big) $$
    and is corresponding to the distribution of the Brownian bridge. Here $H_n$ denote the Hermite polynomials
$$H_n(x)=(-1)^n e^{ \frac{x^2}{2} } \Big(\dfrac{d}{dx} \Big)^n e^{ -\frac{x^2}{2} }. $$
    From other side, define $\mu$ as a distribution in $\mathscr{C}_0\big( [0,1],\R\big)$ of the random process 
    $$\eta(t)=t.\xi,\; t\in [0,1]$$
    where $\xi$ is a standard Gaussian random variable. Then, again, $\Phi_\ve$ converges when $\ve \to 0^+$ to a generalized functional with the same chaotic expansion but with another measure representation via Sugita theorem. Now it is a probability measure concentrated on the one function $f\equiv 0$. Actually, it can be checked that $\Phi_\ve$ converges to a generalized functional for a very wide set of measures $\mu$. In this sense, $\Phi_\ve$ is universal.
\end{example}

  In this paper we will discuss the conditions for universality of the different families $\{ \Phi_\ve, \ve >0 \}$. To do this we use the notion of finite absolute continuity of measures on the Banach space introduced in \cite{AADFAC}. This notion defines connection between the weak moments of two probability measures.
\begin{defn}\cite{AADFAC} \label{FAC0}
A probability measure $\nu$ with weak moments of arbitrary order on the Banach space $B$ is finitely absolutely continuous with respect to the probability measure $\mu$ (this fact is denoted $\nu <<_0 \mu$) if for any $n\in \N$ there exists a constant $c_n>0$ such that for any polynomial $P_n : \R^n\to \R$ with degree at most $n$ and any $\ell_1,\cdots, \ell_n \in B^*$ we have 
$$\Big| \int _{ B} P_n ( \ell_1(\omega),\cdots , \ell_n (\omega) ) d\nu(\omega) \Big| \leqslant c_n \Big(  \int _{ B } P_n^2 ( \ell_1(\omega),\cdots, \ell_n (\omega) ) d\mu(\omega) \Big) ^{\frac{1}{2} }. 
$$ 
\end{defn}
Note that if $\mu$ is a Gaussian measure on $B$ and $\nu$ is a measure related to a positive generalized functional on $\big( B,\mathcal{B},\mu\big)$  then, evidently,
$$\nu <<_0 \mu.$$
This observation leads to the following idea. It seems that uniform condition of finite absolute continuity 
$$\Phi_\ve \mu <<_0 \mu,\; \ve >0$$
will guarantee existence of the generalized functional $\Phi_0$ on $\big( B,\mathcal{B},\mu\big)$ which is a limit of $\Phi_\ve$ when $\ve$ tends to $0$. It occurs to be true. Of course the condition of uniformity can specified at concrete cases. We present the corresponding examples related to the functional counting self-intersections.\\

\begin{example}
Let $w(t), t\in[0,1]$ be a $d$-dimensional Brownian motion and let $u\in \R^d\setminus \{0\}$. Consider the following formal expression 
\begin{equation}
    \label{formal}
    \rho(u)=\int_{\Delta_2} \delta_u\big( w(t_2)-w(t_1) \big) dt_1dt_2
\end{equation}
In \cite{IPA} a rigorous meaning was associated to $\rho(u)$ as follows. First we define the Dirac $\delta$-function $\delta_u$ as a limit of the family $\{ p_\ve ^d, \ve >0 \}$ of functions defined by 
$$p_\varepsilon^d (x)=\dfrac{1}{(2\pi \varepsilon)^{\frac{d}{2}}} \exp \Big( -\dfrac{\|x\|^2}{2\varepsilon} \Big), \quad \varepsilon >0,\, x\in \R^d.  $$
($p_\ve ^1$ will simply be denoted $p_\ve$.)\\
Consequently, $\rho(u)$ should be understood as
\begin{equation}
    \label{nonformal}
    \rho(u)=  \lim_{\varepsilon \to 0}\int_{\Delta_2}   p_\varepsilon ^d \big( w(t_2)-w(t_1)-u \big) dt_1dt_2
\end{equation} 
Denote by $\Phi_\ve$ the integral in the right hand side of \eqref{nonformal}. It was proved (see \cite{IPA}) that $\Phi_\ve$ has an It\^{o}-Wiener expansion given by 
\begin{equation}\label{Chaos1}
\Phi_\ve
=\sum_{k= 0}^{\infty}  \sum _{n_1+\cdots+n_d=k} \int_{\Delta_2} \prod _{1\leqslant j\leqslant d} \left\lbrace  \dfrac{1}{  n_j! }\,H_{n_j}\Big( \dfrac{w_j(t_2)-w_j(t_1)}{\sqrt{ t_2-t_1+\ve} } \Big)\,H_{n_j}\Big( \dfrac{u_j}{\sqrt{ t_2-t_1+\ve} } \Big) \right\rbrace \,  p^d_{t_2-t_1+\ve}(u)  dt_1dt_2
\end{equation} 
Taking the limit when $\ve \to 0$ in each term of \eqref{Chaos1}
yields the formal It\^{o}-Wiener expansion of $\rho(u)$ : 
\begin{equation}\label{Chaos}
\rho(u)
=\sum_{k= 0}^{\infty}  \sum _{n_1+\cdots+n_d=k} \int_{\Delta_2} \prod _{1\leqslant j\leqslant d} \left\lbrace  \dfrac{1}{  n_j! }\,H_{n_j}\Big( \dfrac{w_j(t_2)-w_j(t_1)}{\sqrt{ t_2-t_1} } \Big)\,H_{n_j}\Big( \dfrac{u_j}{\sqrt{ t_2-t_1} } \Big) \right\rbrace \,  p^d_{t_2-t_1}(u)  dt_1dt_2
\end{equation} 
It was proved in \cite{IPA} that the formal series $\rho(u)$ is in fact an element of the Sobolev spaces $\D^{2,\gamma}$ such that 
$$\gamma < \dfrac{4-d}{2} $$
and that $\Phi_\ve$ converges, when $\ve \to 0$, to $\rho(u)$ in each of those spaces.\\

Therefore, the intersection local time formally defined by $\rho(u)$ is a positive generalized Wiener functional whenever $d\geqslant4$. Besides, using Sugita theorem, (see \cite{DS23}), $\rho(u)$ can be represented by a measure on the Wiener space $\mathscr{C}_0\big( [0,1],\R^d\big)$ if, and only if, $d\geqslant 4$.
\end{example}

  Correspondingly to the above mentioned arguments, the article is divided into three parts. The first part contains necessary definitions and facts about finite absolute continuity, mostly from \cite{AADFAC,Riabov}. Second part contains a proof of convergence of $\Phi_\ve$ under uniform finite absolute continuity. The last part contains concrete examples of universal families $\Phi_\ve, \, \ve >0$.

\section{Survey about finite absolute continuity and polynomially non degenerate measures}
We recall here some definitions, examples and statements introduced and analysed mainly by A.A. Dorogovtsev in \cite{AADFAC}.\\
\noindent 
For every $n\in \N$ let $\mathcal{P}_n$ be the set of all polynomials of degree less or equal to $n$ defined on $B$. Denote by $\overline{ \mathcal{P}_n }$ its closure in $L_2\big( B,\mu \big)$. Define $K_n$ as the orthogonal complement of $ \overline{ \mathcal{P}_n } $ in $\overline{ \mathcal{P}_{n+1} }$. Since the set of all finite dimensional polynomials is dense in $L_2\big( B,\mu \big)$ then the following orthogonal decomposition holds 
$$ L_2\big( B,\mu \big)=\bigoplus _{n=0}^\infty K_n. $$
Denote by $J_n$ the orthogonal projection in $L_2\big( B,\mu \big)$ onto $K_n$. Denote by $H_{n,s}$ the space of $n$-linear continuous symmetric forms on $B$. If we denote by $H_2$ the space $K_1$ then the space $H_{n,s}$ can be identified with the symmetric part of the tensor power $H_2^{ \otimes n }$. Denote by $\|\cdot\| _n$ the associated norm. 
\begin{defn}\cite{AADFAC} \label{PNM}
A measure $\mu$ on $B$ is called polynomially non-degenerate if there exist sequences $(c_n)_n$ and $(C_n)_n$ of positive numbers such that, for any finite dimensional $n$-linear symmetric continuous form $A_n$ on $B$, the following inequality holds 
$$c_n \| A_n\| _n^2 \leqslant  \int _{ B} \Big(J_n A_n \Big)^2(\omega) d\mu(\omega)  \leqslant C_n \| A_n\| _n^2. 
$$ 
\end{defn}
The following results give us some examples of polynomially non-degenerate measures.
\begin{lem}\cite{AADFAC}
    Suppose a measure $\mu$ is polynomially non-degenerate and a measure $\nu$ is such that $\mu \sim \nu$ and the following conditions are satisfied 
    \begin{enumerate}
        \item $0< \mathrm{ess \, inf} \dfrac{d\nu}{d\mu} \leqslant \mathrm{ess \, sup} \dfrac{d\nu}{d\mu} < \infty    $
        \item  the mean value of $\nu$ is equal to $0$.
    \end{enumerate}
    Then $\nu$ is polynomially non-degenerate.
\end{lem}
    
\begin{lem}\cite{AADFAC}
    Suppose that $B$ is a real separable Hilbert space and $\mu$ is a Gaussian measure on $B$ with mean value zero and the non degenerate correlation operator whose eigenvalues $\{ \lambda_n, n\geqslant 0 \}$  are such that 
    $$ \sum_{n=1}^\infty \lambda_n \log^2 (n) < \infty .$$
    Then the measure $\nu$, obtained from $\mu$ by restriction to the ball $B(0,r)$ of center 0 and radius $r$ and by normalization, is polynomially non-degenerate.
\end{lem}

Now we focus on the notion of finite absolute continuity. Definition \eqref{FAC0} precises the meaning of this concept. In the remainder of this section we assume that the reference measure $\mu$ is polynomially non-degenerate. One of the main examples of finite absolutely continuous measures is provided by the following result.
\begin{thm}\cite{Riabov}
Let $\eta \in \D^{2,-\infty}$ be a positive generalized Wiener function and $\nu $ be the measure on the Wiener space $\mathscr{C}_0\big( [0,1],\R^d\big)$ associated to $\eta$. Then $\nu$ is finitely absolutely continuous with respect to the Wiener measure $\mu_0$.
\end{thm}
Finite absolute continuity may imply absolute continuity in certain cases.
\begin{example}\cite{AADFAC}
    The sequence $(c_n)_{n\geqslant0}$ in definition \eqref{FAC0} can be chosen bounded if, and only if, $\nu << \mu$ and 
    $$ \dfrac{d\nu}{d\mu} \in L_2\big( B,\mu \big). $$
\end{example}

\begin{lem}\cite{AADFAC}
Suppose that $\nu$ and $\mu$ are Gaussian measures in a real separable Hilbert space $B$ that have the same correlation operator $S$ and mean values $0$ and $h$, respectively. If $\nu <<_0 \mu$ then $\nu << \mu$.
\end{lem}

If a measure $\nu$ is finitely absolutely continuous with respect to $\mu$ then it is possible to obtain a chaotic expansion of $\nu$ with respect to $\mu$ in the sense precised by the following theorem.

\begin{thm}\cite{AADFAC}
    Let $\nu$ be a probability measure on $B$ which is finitely absolutely continuous with respect to $\mu$.  Then there exists a sequence of kernels $A_n\in H_{n,s}$ such that for any polynomial $Q$ defined on $B$ the following equality holds 
    $$ \int_B Q(\omega) \nu (d\omega) =\sum_{n=0}^\infty \int_B Q(\omega)A_n(\omega) \mu (d\omega) . $$
\end{thm}

\section{Existence of the generalized functional}
In this section we present conditions under which the family $\Phi_\ve$ has a limit which is a generalized functional. Begin with the statement that uniform finite absolute continuity of family of probability measures leads to weak compactness of this family. We consider two cases. When $B$ is a Hilbert space or the space $\mathscr{C}\big( [0,1],\R^d\big)$. Such choice is enough for our purposes. Let us start from the Hilbert space. Suppose that $\mu$ is polynomially non-degenerate centered probability measure on the Hilbert space $H$ such that 
$$ \forall \; h\in H \,:\; \int_H (h,u)^2\mu(du) >0. $$
\begin{thm}\label{thm1}
    Let the family of probability measures $\{ \mu_\alpha,\, \alpha \in \Theta \}$ on $H$ satisfies conditions 
    \begin{enumerate}
        \item for every $\alpha \in \Theta$, $\mu_\alpha$ has all moments
        \item for every $\alpha \in \Theta$, $\mu_\alpha <<_0 \mu$ with the constants $\{ c_n^\alpha,\, n\geqslant0 \}$
        \item for every $n \geqslant 0$ 
        $$ \sup _{\alpha \in \Theta} c_n^\alpha :=c_n <\infty. $$
    \end{enumerate}
    Then the family $\{ \mu_\alpha,\, \alpha \in \Theta \}$ is weakly compact.
\end{thm}
\begin{proof}
    Since $\mu$ is polynomially non-degenerate then there exists a constant $c>0$ such that
    $$\forall \; h\in H,\quad  \int_H (h,u)^4\mu(du) \leqslant c^2 \Big( \int_H (h,u)^2\mu(du) \Big)^2. $$
Let $S$ be the covariance operator of the measure $\mu$. Denote by $\{e_k;\, k\geqslant 1\}$ the orthonormal eigenbasis of $S$. Then, for every $\alpha \in \Theta$ 
\begin{equation}\label{Pro}
    \sum_{k=1}^\infty \int_H (e_k,u)^2\mu_\alpha(du)\leqslant c_2^\alpha \sum_{k=1}^\infty \sqrt{\int_H (e_k,u)^4\mu(du) }\leqslant c c_2^\alpha \sum_{k=1}^\infty \int_H (e_k,u)^2\mu(du)
\end{equation}
    In the same way it can be checked that 
    \begin{equation}
        \label{Prof}
        \sup_{\alpha \in \Theta} \sum_{k=n}^\infty \int_H (e_k,u)^2\mu_\alpha(du) \xrightarrow[ n \to \infty]{} 0
    \end{equation}
    It is known \cite{Araja} that \eqref{Pro} and \eqref{Prof} are sufficient for the weak compactness of $\{ \mu_\alpha,\, \alpha \in \Theta \}$. \\
    Theorem is proved.
\end{proof}
\begin{remark}
    As it can be seen from the proof, only second and fourth moments were used but we formulate the theorem in a frame of our main considerations.
\end{remark}

Now consider the space $B=\mathscr{C}\big( [0,1],\R^d\big)$. Suppose that the measure $\mu$ is centered and non-degenerated. 

\begin{thm}
    Suppose that the family $\{ \mu_\alpha,\, \alpha \in \Theta \}$ satisfies conditions of Theorem \eqref{thm1} and $\mu$ satisfies the condition
    \begin{equation}
        \exists \, c>0,\, \exists \, \gamma >0 \, : \; \forall \; t_1,t_2 \in [0,1] \, : \; \int_B \big \| u(t_2)-u(t_1) \big \| ^2 \mu(du) \leqslant c \big | t_2-t_1 \big| ^\gamma
    \end{equation}
    Then the family $\{ \mu_\alpha,\, \alpha \in \Theta \}$ is weakly compact in $B$.
\end{thm}
\begin{proof}
    Similarly to the previous proof it can be checked that for an arbitrary $m\geqslant 1$ there exists $a_m>0$ such that for every $\varphi \in B^*$ 
    $$ \int_B \varphi(u)^{2m} \mu(du) \leqslant a_m \Bigg( \int_B \varphi(u)^2\mu(du) \Bigg) ^m. $$
    Then it follows from the condition of the theorem that for some $m_0\geqslant 1$ 
    $$ \int_B \big \| u(t_2)-u(t_1) \big \| ^{2m_0} \mu(du) \leqslant a_{m_0} c^{m_0}  \big | t_2-t_1 \big| ^{1+\beta} $$
    where $\beta>0$. Now from uniform finite absolute continuity one can conclude that 
    \begin{equation} \label{abcd} \sup _{\alpha \in \Theta } \int_B \big \| u(0) \big \| ^{2} \mu_\alpha(du) < \infty  
    \end{equation}
    and for some $D >0$ 
    \begin{equation}
        \label{99}
        \forall \; t_1,t_2 \in [0,1] \, : \; \sup _{\alpha \in \Theta } \int_B \big \| u(t_2)-u(t_1) \big \| ^{2m_0} \mu_\alpha (du) \leqslant D \big | t_2-t_1 \big| ^{1+\beta} 
    \end{equation}
    It is known \cite{Bill} that \eqref{abcd} and \eqref{99} are sufficient for weak compactness of $\{ \mu_\alpha,\, \alpha \in \Theta \}$. \\
    Theorem is proved.
\end{proof}

\begin{remark}
    In view of weak compactness, to check the existence of the limiting generalized functional is not difficult. If all the integrals with finite-dimensional polynomials for $\Phi_\ve$ converge when $\ve$ tends to $0$ then the weak limit $\Phi_0$ of $\Phi_\ve$ exists. Due to uniform finite absolute continuity, $\Phi_0$ is also finitely absolutely continuous with respect to initial measure. Hence in the next section we do not repeat these arguments at the end of every example.
\end{remark}

\section{Examples of universal families}

In this section we consider concrete approximating families and their limits. Let us start from the generalized local time. 
\begin{example} \label{Ex1}
    
Suppose that $\xi(t),\, t\in [0,1]$ is a centered Gaussian process such that 
\begin{enumerate}
    \item $\forall \, t\in [0,1],\; \E \xi^2(t)=\sigma ^2(t)>0$
    \item $\exists \, c>0,\, \exists \, \gamma >0, \; \forall \; t_1,t_2 \in [0,1] \, : \; \E \big | \xi(t_2)-\xi(t_1) \big | ^2  \leqslant c \big | t_2-t_1 \big| ^\gamma$.
\end{enumerate}
Note that under these conditions $\xi$ has a continuous modification. Consequently, the distribution $\mu$ of $\xi$ is a centered Gaussian measure in $\mathscr{C}\big( [0,1],\R\big)$. We recall that every Gaussian measure is polynomially non-degenerate. Now consider for $\ve >0$ the functional on  $\mathscr{C}\big( [0,1],\R\big)$ defined by 
$$ \Phi_\ve (f)=\int_0^1 p_\ve  \big( f(t) \big) dt. $$
Let us check that the measures $\mu_\ve =\Phi_\ve \mu$ are uniformly finitely absolutely continuous with respect to $\mu$. Consider a finite dimensional polynomial $P_n$ of degree $n$ on $\mathscr{C}\big( [0,1],\R\big)$. Then 
\begin{align*}
   \Big| \int_{ \mathscr{C}\big( [0,1],\R\big)} \Phi_\ve(f)P_n(f) \mu(df) \Big| 
   &= \Big|\E \big( \Phi_\ve(\xi) P_n(\xi) \big)      \Big|\\
   &=\Big|\E  \int_0^1 p_\ve \big( \xi(t) \big) dt \, P_n(\xi)   \Big|\\
   &=\Big|  \int_0^1  \E \Big[ p_\ve \big( \xi(t) \big) \, \E \big(  P_n(\xi) \big| \xi(t) \big) \Big]   dt   \Big|\\
   &\leqslant \int_0^1 \E \Big[ p_\ve  \big( \xi(t) \big) \, \Big|\E \big(  P_n(\xi) \big| \xi(t) \big)\Big| \Big]   dt   \\
   &\leqslant \int_0^1 \sup_{ \R} p_\ve   *  \Big|   \big( p_{ \sigma^2(t) } \cdot Q_n\big) (\cdot ,t)\Big|dt.
\end{align*}
Here 
$$ Q_n(x,t)= \E \Big(  P_n(\xi) \Big| \xi(t)=x \Big) $$
is a polynomial of degree not greater than $n$. Using Newton-Leibniz formula and Cauchy inequality it can be checked that, when
$$ \min _{t\in[0,1]} \sigma^2(t) >0,$$
\begin{align*}
    \max_{\R} \Big | p_{ \sigma^2(t) } \cdot Q_n \Big| 
    & \leqslant c\, \sqrt{ \int_\R p_{ \sigma^2(t) }(x) \cdot Q_n^2(x) dx   } \\
    &= c \, \sqrt{ \E \Bigg[  \Big( \E \big(  P_n(\xi) \big| \xi(t) \big) \Big)^2 \Bigg] }\\
    &\leqslant c\, \sqrt{ \E \Big( P_n^2(\xi) \Big) }.
\end{align*} 
This finishes the proof of uniform finite absolute continuity.

\end{example}

\begin{example} \label{bbb}
Let $w(t)=(w_1(t),\cdots,w_d(t)) $ be a $d$-dimensional Brownian motion and $A$ an invertible bounded linear operator in the Hilbert space $ L_{2}([0,1]) $. Define a multidimensional Gaussian integrator by
\begin{equation}
\label{integrator}
     X(t)=\big( X_1(t),\cdots,X_d(t) \big),\quad X_j(t)=\int_0^1 \big( A \mathbf{1}_{[0,t]} \big) (s) dw_j(s),\; j=1,\cdots,d
\end{equation}
The distribution $\mu$ of $X$ is a centered Gaussian measure in $B=\mathscr{C}_0\big( [0,1],\R^d\big)$. Consider for $\ve >0$ the functional on  $B$ defined by 
$$ F_\ve (f)=\int_0^1 p_\ve^d  \big( f(t)-u \big) dt,\, u\in \R^d\setminus \{0\}.$$
Let us check that the measures $\mu_\ve =F_\ve \mu$ are uniformly finitely absolutely continuous with respect to $\mu$. Consider a finite dimensional polynomial $P_n$ of degree $n$ on $B$. Similarly to Example  \eqref{Ex1} one can check that 
$$\Big| \int_{B} F_\ve(f)P_n(f) \mu(df) \Big| 
   \leqslant \int_0^1 \E \Big[ p_\ve ^d \big( X(t) -u\big) \, \Big|\E \big(  P_n(X) \big| X(t) \big)\Big| \Big]   dt   \leqslant \int_0^1  p^d _\ve  *  \Big|     p_{ \sigma^2(t)}^d \cdot Q_n (\cdot ,t)\Big|(u)dt $$
where 
$$ Q_n(x,t)= \E \Big(  P_n(X) \Big| X(t)=x \Big) $$
is a polynomial of degree not greater than $n$, and 
$$ \sigma(t)=\|A\mathbf{1}_{[0,t]} \|. $$
Since $A$ is invertible then there exist $0<m<M$ such that
\begin{equation} \label{aaaa}
    \sqrt{m}\sqrt{t} \leqslant \sigma(t) \leqslant \sqrt{M} \sqrt{t}
\end{equation} 
Let us recall that for $\sigma^2 >0$ the Hilbert space $L_2\Big( \R^d, p_{\sigma^2}^d(x)dx\Big) $ has an orthonormal basis : 
$$ R_{n_1,\cdots,n_d}(x)=\sigma^{n_1+\cdots+n_d} \prod_{j=1}^d H_{n_j} \Big( \dfrac{x_j}{\sigma} \Big),\; n_1,\cdots,n_d \in \N. $$
Consequently, 
$$ Q_n(x,t)=\sum_{n_1+\cdots+n_d\leqslant n} \alpha_{n_1,\cdots,n_d}(t)R_{n_1,\cdots,n_d}(x)  $$
and therefore

\begin{align*}
 \Big | p_{ \sigma^2(t) }^d(x) \cdot Q_n(x,t) \Big|
    &=\Big | p_{ \sigma^2(t) }^d (x)\; \sum_{n_1+\cdots+n_d\leqslant n} \alpha_{n_1,\cdots,n_d}(t)R_{n_1,\cdots,n_d}(x) \Big|\\
    & \leqslant p_{ \sigma^2(t) }^d (x) \;\sqrt{ \sum_{n_1+\cdots+n_d\leqslant n} \alpha_{n_1,\cdots,n_d}(t)^2 }  \;\sqrt{ \sum_{n_1+\cdots+n_d\leqslant n} R_{n_1,\cdots,n_d}(x)^2 } \\
     & = p_{ \sigma^2(t) }^d (x) \Big \| Q_n(\cdot,t)\Big \| _{ L_2\big( \R^d, p_{\sigma^2(t)}^d(x)dx\big) }  \sqrt{ \sum_{n_1+\cdots+n_d\leqslant n} R_{n_1,\cdots,n_d}(x)^2 } \\
    &= \sqrt{ \E \Bigg[  \Big( \E \big(  P_n(X) \big| X(t) \big) \Big)^2 \Bigg] } \; p_{ \sigma^2(t) }^d (x) \sqrt{ \sum_{n_1+\cdots+n_d\leqslant n} R_{n_1,\cdots,n_d}(x)^2 }  \\
    &\leqslant \big( \max \{\sqrt{M}, 1 \} \big) ^n\; \sqrt{ \E \Big( P_n^2(X) \Big) }  \; p_{ \sigma^2(t) }^d (x) \sqrt{ \sum_{n_1+\cdots+n_d\leqslant n} \prod_{j=1}^d H_{n_j} ^2 \Big( \dfrac{x_j}{\sigma(t)} \Big) }.
\end{align*} 
Now we use the following inequality \cite{DOGS}
\begin{equation}
\label{Hermita}
\forall \; n\in \N,\;  \exists \, a_n>0\, ; \; \forall \,x\in \R\; ; |H_n(x)| \leqslant a_n e^{\frac{x^2}{4}}
\end{equation}
Consequently, there exists $C_n>0$ such that 
\begin{align*}
 \Big | p_{ \sigma^2(t) }^d(x) \cdot Q_n(x,t) \Big|
   &\leqslant C_n \sqrt{ \E \Big( P_n^2(X) \Big) } \; p_{ 2\sigma^2(t) }^d (x) .
\end{align*} 
Therefore,

   $$\Big| \int_{B} F_\ve(f)P_n(f) \mu(df) \Big| 
   \leqslant  C_n \sqrt{ \E \Big( P_n^2(X) \Big) } \int_0^1  p^d _\ve  *  p_{ 2\sigma^2(t) }^d (u)dt=  C_n \sqrt{ \E \Big( P_n^2(X) \Big) } \int_0^1    p_{\ve+ 2\sigma^2(t) }^d (u)dt\,.
$$
Using \eqref{aaaa}  one can find a constant $c(m,M)>0$ such that $$ p_{\ve+ 2\sigma^2(t) }^d (u) \leqslant C(m,M) p_{\ve+ 2Mt }^d (u) . $$ 
Moreover, for any $\ve \in ]0,1]$ 
\begin{align*}
    \int_0^1    p_{\ve+ 2Mt }^d (u)dt 
    &=\int_0^{\ve}   p_{\ve+ 2Mt}^d (u)dt + \int_\ve ^1   p_{\ve+ 2Mt }^d (u)dt \\
    & \leqslant \dfrac{\exp \Big(  -\dfrac{ \|u\|^2 }{2(1+2M)\ve } \Big)  }{(2\pi) ^{\frac{d}{2} }\ve ^{\frac{d}{2} -1 }  } + \Big( \dfrac{1+2M}{2M} \Big)^{ \frac{d}{2} }  \int_\ve ^1   p_{(1+ 2M)t }^d (u)dt \\
    & \leqslant \dfrac{\exp \Big(  -\dfrac{ \|u\|^2 }{2(1+2M)\ve } \Big)  }{(2\pi) ^{\frac{d}{2} }\ve ^{\frac{d}{2} -1 }  } + \Big( \dfrac{1+2M}{2M} \Big)^{ \frac{d}{2} }  \int_0 ^1   p_{(1+ 2M)t }^d (u)dt \\
    & \leqslant \kappa
\end{align*}
where the constant $\kappa$ depends only on $M$ and $d$. Finally, there exists a constant $c_n$ such that 
$$  \Big| \int_{B} F_\ve(f)P_n(f) \mu(df) \Big| 
   \leqslant  c_n \sqrt{ \E \Big( P_n^2(X) \Big) } \,. $$
which proves uniform finite absolute continuity of the family $F_\ve$. \\

As a consequence of this fact, one can say that the local time of the multidimensional integrator \eqref{integrator} at any point different from the origin exists as a generalized function and admits an expansion into a series of multiple stochastic integrals with respect to the integrator itself.

\end{example}

\begin{example}
We continue with the multidimensional Gaussian integrator defined by \eqref{integrator}. Let $\mu$ be the distribution of $X$ in the Banach space $B=\mathscr{C}_0\big( [0,1],\R^d\big)$. Consider for $\ve >0$ the functional defined on  $B$ by 
$$G_\ve (f)=\int_{\Delta_2} p_\ve^d  \big( f(t)-f(s) -u\big) dsdt,\, u\in \R^d\setminus \{0\}.$$
Let us check that the measures $\mu_\ve =G_\ve \mu$ are uniformly finitely absolutely continuous with respect to $\mu$. Consider a finite dimensional polynomial $P_n$ of degree $n$ on $B$. Similarly to Example  \eqref{bbb} one can check that 
\begin{align*}
    \Big| \int_{B} G_\ve(f)P_n(f) \mu(df) \Big| 
   &\leqslant \int_{\Delta_2} \E \Big[ p_\ve ^d \big( X(t) -X(s)-u\big) \, \Big|\E \big(  P_n(X) \big| X(t)-X(s) \big)\Big| \Big]   dsdt \\
   &\leqslant \int_{\Delta_2}  p^d _\ve  *  \Big|     p_{ \sigma^2(s,t)}^d \cdot Q_n (\cdot ,s,t)\Big|(u)dsdt 
\end{align*}
where 
$$ Q_n(x,s,t)= \E \Big(  P_n(X) \Big| X(t)-X(s)=x \Big) $$
is a polynomial of degree not greater than $n$, and 
$$ \sigma(s,t)=\|A\mathbf{1}_{[s,t]} \|. $$
Since $A$ is invertible then there exist $0<m<M$ such that
\begin{equation} \label{bbbb}
    \sqrt{m}\sqrt{t-s} \leqslant \sigma(s,t) \leqslant \sqrt{M}\sqrt{t-s}
\end{equation} 
Using again the orhtonormal basis of $L_2\Big( \R^d, p_{\sigma^2}^d(x)dx\Big) $ one can find a constant $C_n>0$ such that 
$$\Big | p_{ \sigma^2(s,t) }^d(x) \cdot Q_n(x,s,t) \Big|
   \leqslant C_n \sqrt{ \E \Big( P_n^2(X) \Big) } \; p_{ 2\sigma^2(s,t) }^d (x) .$$
Thus,
\begin{align*}
   \Big| \int_{B} G_\ve(f)P_n(f) \mu(df) \Big| 
   &\leqslant  C_n \sqrt{ \E \Big( P_n^2(X) \Big) } \int_{\Delta_2}  p^d _\ve  *      p_{ 2\sigma^2(s,t)}^d (u)dsdt \\
   &=  C_n \sqrt{ \E \Big( P_n^2(X) \Big) } \int_{\Delta_2}    p_{\ve+ 2\sigma^2(s,t) }^d (u)dsdt\\
   &\leqslant  c(m,M) C_n \sqrt{ \E \Big( P_n^2(X) \Big) } \int_{\Delta_2}    p_{\ve+ 2M(t-s) }^d (u)dsdt\\
   &\leqslant  c(m,M) C_n \sqrt{ \E \Big( P_n^2(X) \Big) } \int_{0}^1    p_{\ve+ 2Mt }^d (u)dt\\
   &\leqslant  \underbrace{\kappa c(m,M) C_n }_{c_n}\sqrt{ \E \Big( P_n^2(X) \Big) }
\end{align*}
This proves uniform finite absolute continuity of the family $G_\ve$.\\

From this point, one can deduce that the intersection local time of the multidimensional integrator \eqref{integrator}, formally defined by,
$$ \rho_X(u)=\int_{\Delta_2} \delta_u \big( X(t)-X(s) \big) dsdt= \lim _{\ve \to 0} G_\ve (X),\, u\in \R^d\setminus \{0\} $$
is a generalized function from $X$ and admits a chaos expansion with respect to $X$.
\end{example}

\end{document}